\documentclass{amsart}
\usepackage{amsmath}
\usepackage{amssymb}
\usepackage{amsthm}

\usepackage[colorlinks]{hyperref}

\newtheorem{theorem}{Theorem}[section]
\newtheorem{lemma}[theorem]{Lemma}
\newtheorem{proposition}[theorem]{Proposition}
\newtheorem{corollary}[theorem]{Corollary}

\theoremstyle{definition}
\newtheorem{definition}[theorem]{Definition}

\theoremstyle{remark}
\newtheorem{remark}[theorem]{Remark}

\numberwithin{equation}{section}

\def\e{\varepsilon}

\begin{document}

\title[Stability under Hausdorff convergence]{Stability of metric viscosity solutions under Hausdorff convergence}

\author{Shimpei~Makida}
\address{Department of Mathematics, Graduate School of Science, Hokkaido University, North
10, West 8, Kita-Ku, Sapporo, 060-0810 Japan}
% \curraddr{}
\email{makida.shimpei.k3@elms.hokudai.ac.jp}

\author{Atsushi~Nakayasu}
\address{Department of Mathematics, Faculty of Science, Kyoto University, Kitashirakawa Oiwake-cho, Sakyo-ku, Kyoto, 606-8502 Japan}
% \curraddr{}
\email{ankys@math.kyoto-u.ac.jp}

% 35D40 Viscosity solutions to PDEs
% 35F21 Hamilton-Jacobi equations
% 35B35 Stability in context of PDEs


\date{\today}

%\dedicatory{}

%\commby{}

\begin{abstract}
This study investigated the stability of Hamilton--Jacobi equation on general metric spaces with a perturbation in some whole space. 
This type of stability appears in the domain perturbation problem.
We find that the stability holds when the set converges in the Hausdorff sense and when the metric converges in some uniform sense.
Examples of the perturbed space satisfying these assumptions include network approximation of self-similar sets such as the Sierpi\'{n}ski gasket, junction of shrinking tubes, and lattice lines with the Manhattan distance.
We also give supplemental results on time-dependent or noncompact case.
Stability can be achieved when the class of test function of metric viscosity solutions is reduced to the squared distance functions, whose proof is also given.
\end{abstract}

\maketitle

\section{Introduction}
\label{s:intro}

Consider the Hamilton--Jacobi equation of the generalized form
\begin{equation}
\label{e:fhj}
H(x, u(x), |\nabla_d u|(x)) = 0 \quad \text{for $x \in X$.}
\end{equation}
Here, $(X, d)$ is a complete geodesic space embedded in a whole space $(\mathsf{X}, \mathsf{d})$.
The symbol $|\nabla_d u|$ represents a so-called \emph{local slope} of an unknown function $u: X \to \mathbb{R}$:
\begin{equation}
\label{e:slope}
|\nabla_d u|(x) := \limsup_{y \to x} \frac{|u(y)-u(x)|}{d(x, y)} \quad \text{for $x \in X$.}
\end{equation}
The local slope is well defined as a non-negative real number in $\mathbb{R}_+ := [0, \infty)$ as far as $u$ is locally Lipschitz continuous.
$H$ is a real-valued function on $X\times\mathbb{R}\times\mathbb{R}_+$ called a Hamiltonian.

This study mainly evaluated the \emph{stability} of the Hamilton-Jacobi equation \eqref{e:fhj} with respect to the space $(X, d)$.
That is, we study the sequence of solutions of the Hamilton-Jacobi equation on a variable space $(X_n, d_n)$ in $(\mathsf{X}, \mathsf{d})$ of the form
\begin{equation}
\label{e:hj}
H_n(x, u^n(x), |\nabla_{d_n} u^n|(x)) = 0 \quad \text{for $x \in X_n$}
\end{equation}
with the indexes $n \in \mathbb{N}$
and show that the solution $u^n$ converges to a solution of the limit equation
\begin{equation}
\label{e:lhj}
H_\infty(x, u^\infty(x), |\nabla_{d_\infty} u^\infty|(x)) = 0 \quad \text{for $x \in X_\infty$}
\end{equation}
under a suitable convergence of $X_n$, $d_n$ and $H_n$.

The problem of the stability with respect to domains on which the equation is considered is strongly correlated to the problems in domain perturbation and homogenization, which has been extensively studied.
Due to the large number of results, we only mention those relating to the Hamilton--Jacobi equation.
The stability of viscosity solutions of the Hamilton-Jacobi equations with respect to a sequence of domains appears in the proof of existence of the solutions by domain approximation.
In \cite{AT15}, the solution of the Hamilton-Jacobi equation on the junction is constructed as the graphical limit of the optimal control in the region where the junction is fattened in two dimensions.
In \cite{CCM16}, the viscosity solution of the eikonal equation on the Sierpi\'{n}ski gasket is constructed using the graphical approximation of the Sierpi\'{n}ski gasket.
To generalize the results of \cite{CCM16}, we investigated the stability of viscosity solutions of the Hamilton-Jacobi equations when the increasing sequence of networks in a Euclidean space converges in the Hausdorff sense \cite{M23}.

This study was aimed to extend the results of the previous work \cite{M23} and to investigate for which types of a sequence of metric spaces the stability holds.
In \cite{M23}, the uniform Lipschitz estimate of the family of viscosity solutions is required to extract the limit function at the Hausdorff limit.
In this study, the relaxed half semilimit was used rather than the uniform Lipschitz estimate.  
The relaxed half semilimit is a method used to study large time behavior and homogenization, and the definition itself does not require any regularity property of the solution.
This allows us to consider the stability of viscosity solutions for a sequence of metric spaces that are not necessarily monotonically increasing.

The stability of the heat equation has been studied in the celebrated paper \cite{GMS15}.
In comparison, the Hamilton--Jacobi equation strongly reflects the metric structure of the space, confirming that the heat equation requires the measure structure to the space, while the Hamilton--Jacobi equation requires only the metric structure.
This makes our argument much simpler than the argument in the heat equations.

Similar in \cite{M23}, we adopted the metric viscosity solutions of Gangbo and {\'S}wi{\c{e}}ch \cite{GS14,GS15} as a notion of solutions of the Hamilton-Jacobi equations.
The metric viscosity solutions is the metric space version of the viscosity solutions to deal with fractals and networks, introduced in \cite{GHN15,N14,AF14,GS14,GS15}.
In particular, the metric viscosity solution in \cite{GS14,GS15} is very useful because the comparison theorem holds for a wide class of Hamiltonians.
On equivalence between metric viscosity solutions and other weak solutions of the eikonal equation, we refer the readers to \cite{LSZ21}.

Let us recall the essence of the metric solution solution of Gangbo and {\'S}wi{\c{e}}ch.
The idea is the same as in the classical theory of viscosity solutions.
Consider the situation where the unknown function $u$ is touched by a smooth ($C^{1, \pm}$) function $\phi$ called a test function at a point.
Then, we use the local slope of $\phi$ instead of the one of $u$ in \eqref{e:fhj} at the touching point $\hat{x}$.
The precise definition will be given in Section \ref{s:sol}.
Under such a setting, the unique existence theorem of the solution is shown,
and the stability of standard type holds \cite{NN18}, where standard stability means that the space $(X, d)$ does not perturb.

However, the argument of standard stability cannot be applied to the stability with the space perturbation directly.
The difficulty is that since both the touching point $\hat{x}$ and the test function $\phi$ may be perturbed, we should investigate the convergence of
\begin{equation}
\label{e:convslope}
|\nabla_{d_n}\phi_n|(x_n) \to |\nabla_{d_\infty}\phi|(\hat{x}).
\end{equation}
The critical idea of the present paper to overcome this difficulty is to restrict the class of test functions to the squared distance function
$$
\phi(x) = \frac{k}{2}d(\hat{a}, x)^2
$$
with $k \ge 0$.
Then, condition \eqref{e:convslope} can be reduced to 
$$
d_n(a_n, x_n) \to d_\infty(\hat{a}, \hat{x}).
$$
In order to restrict the class of test functions, Proposition \ref{t:sdtest} is established using a so-called doubling variable method.
A similar technique appears in \cite{N23}.

In view of this argument we show the stability of the metric viscosity solutions under the assumptions
\begin{enumerate}
\renewcommand{\labelenumi}{(H1)}
\renewcommand{\theenumi}{(H1)}
\item
\label{i:h1}
$X_n \to X_\infty$ in $\mathsf{d}$-Hausdorff sense:\
$$
\mathsf{d}_H(X_\infty, X_n) = \max\left\{ \sup_{x \in X_n}\mathsf{d}(X_\infty, x), \sup_{x \in X_\infty}\mathsf{d}(X_n, x) \right\} \to 0,
$$
where $\mathsf{d}(A, x) := \inf_{a \in A}\mathsf{d}(a, x)$ is a distance function from a subset $A$ of $\mathsf{X}$.
\renewcommand{\labelenumi}{(H2)}
\renewcommand{\theenumi}{(H2)}
\item
\label{i:h2}
For arbitrary sequences $a_n \in X_n \to a_\infty \in X_\infty$ and $b_n \in X_n \to b_\infty \in X_\infty$ under the metric $\mathsf{d}$,
the convergence
$$
d_n(a_n, b_n) \to d_\infty(a_\infty, b_\infty)
$$
holds.
\renewcommand{\labelenumi}{(H3)}
\renewcommand{\theenumi}{(H3)}
\item
\label{i:h3}
For arbitrary sequences $x_n \in X_n \to x_\infty \in X_\infty$ under the metric $\mathsf{d}$, $u_n \in \mathbb{R} \to u_\infty$ and $p_n \in \mathbb{R}_+ \to p_\infty$,
the convergence
$$
H_n(x_n, u_n, p_n) \to H_\infty(x_\infty, u_\infty, p_\infty)
$$
holds.
\end{enumerate}
The crucial assumption \ref{i:h2} means that the sequence of the distance functions $d_n$ locally uniformly converges.

\begin{remark}
Conditions \ref{i:h1} and \ref{i:h2} are independent.
For example, consider the arcs
$$
X_n = \{ (\cos\theta, \sin\theta) \mid 0 \le \theta \le 2\pi-n^{-1} \},
\quad X_\infty = \{ (\cos\theta, \sin\theta) \mid 0 \le \theta \le 2\pi \}
$$
embedded in the plane $\mathbb{R}^2$.
Then, condition \ref{i:h2} does not hold.
We remark that the geodesic distance between $(0, 0)$ and $(\cos(2\pi-n^{-1}), \sin(2\pi-n^{-1}))$ over $X_n$ is given by $2\pi-n^{-1}$
while the geodesic distance over $X_\infty$ is $n^{-1}$.
This example satisfies condition \ref{i:h1}.
An example which satisfies \ref{i:h2} and does not fulfill \ref{i:h1} will be given in Remark \ref{t:noth1}.
\end{remark}

Now, which type of sequence of metric spaces satisfies condition \ref{i:h2}?
In this study, we give three main examples.
First is a pre-fractal approximation of self-similar sets.
Generally, since it is difficult to find a solution of the Hamilton--Jacobi equation on fractals, our approach will allow spatial discretization.
However, every fractal and approximation does not satisfy condition \ref{i:h2}.
We show a positive result for the Vicsek fractal and the Sierpi\'{n}ski gasket.
However, the Koch curve does not satisfy condition \ref{i:h2}.

Second is shrinking tubes with a junction as a singularly perturbed problem \cite{AT15}.

Third is the shrinking lattice lines on a Euclidean plane.
The lattice lines converges to the plane in the Hausdorff sense
and the limit metric becomes the so-called Manhattan distance reflecting the intrinsic metric on the lattice lines.

To investigate the stability for time-dependent problems or noncompact metric spaces, our stability theory can be immediately applied to the time-dependent equation.
In the noncompact case, the key tools of the proof of stability is the Ekeland$'$s variational principle \cite{E79}, which corresponds to the maximum principle in general metric spaces.
In this study, we give the stability in the noncompact case on the basis of the discussion in \cite{M23, NN18}.

In section \ref{s:sol}, we recall the definition of the metric viscosity solution and discuss the equivalent condition critical to prove the stability.
In Section \ref{s:stability}, we establish the stability of the metric viscosity solutions in the locally compact case.
In Sections \ref{s:ifs}, \ref{s:junc}, and \ref{s:lattice}, we will give examples satisfying assumption \ref{i:h2} of the stability.
In Section \ref{s:time}, we consider the stability of metric viscosity solutions of time-dependent problems.
In Section \ref{s:noncpt}, we discuss the generalization of the stability to the case when the metric spaces is not necessarily locally compact.

\section{Metric viscosity solutions}
\label{s:sol}

The theory of metric viscosity solutions to the Hamilton--Jacobi equations must be reviewed on a general geodesic space \eqref{e:fhj} proposed by Gangbo and {\'S}wi{\c{e}}ch.
See \cite{GS14, GS15}.
Let $(X, d)$ be a complete geodesic space.

A central notion in this theory is the local slope as defined in \eqref{e:slope}.
The \emph{upper and lower local slopes} are introduced by
$$
|\nabla_d^\pm u|(x) := \limsup_{y \to x} \frac{[u(y)-u(x)]_\pm}{d(x, y)},
$$
where $[\cdot]_\pm$ is the plus or minus part: $[a]_\pm = \max\{ \pm a, 0 \}$.
A locally Lipschitz continuous function $u$ is called $C^{1, \pm}$ if
$|\nabla_d u|$ is continuous and
$$
|\nabla_d^\pm u| = |\nabla_d u|
$$
holds on its domain.

The following lemmas are important.

\begin{lemma}[Maximum principle]
\label{t:mp}
Let $u, v, w$ be locally Lipschitz continuous functions.
\begin{itemize}
\item
If $u-v$ attains a local maximum at $\hat{x}$ and $v$ is $C^{1, -}$, then
$$
|\nabla_d u|(\hat{x}) \ge |\nabla_d v|(\hat{x})
$$
holds.
\item
If $u-v$ attains a local maximum at $\hat{x}$ and $u$ is $C^{1, +}$, then
$$
|\nabla_d u|(\hat{x}) \le |\nabla_d v|(\hat{x})
$$
holds.
\item
If $u-v+w$ attains a local maximum at $\hat{x}$ and $u$ is $C^{1, +}$ and $v$ is $C^{1, -}$, then
$$
||\nabla_d u|(\hat{x})-|\nabla_d v|(\hat{x})| \le |\nabla_d w|(\hat{x})
$$
holds.
\item
If $u-v$ attains a local maximum at $\hat{x}$ and $u$ is $C^{1, +}$ and $v$ is $C^{1, -}$, then
$$
|\nabla_d u|(\hat{x}) = |\nabla_d v|(\hat{x})
$$
holds.
\end{itemize}
\end{lemma}

\begin{lemma}[Squared distance function]
For $a \in X$ the function
$$
\phi(x) = \frac{1}{2}d(a, x)^2
$$
is $C^{1, -}$
and
$$
|\nabla_d^-\phi|(x) = |\nabla_d\phi|(x) = d(a, x)
$$
for all $x \in X$.
\end{lemma}

Since the proofs of these lemmas are straightforward,
we omit them.

A locally Lipschitz function $\psi(x) = \psi_1(x)+\psi_2(x)$ is a \emph{test function for subsolutions}
if $\psi_1 \in C^{1, -}$ and $\psi_2$ is a locally Lipschitz function and \emph{for supersolutions}
if $\psi_1 \in C^{1, +}$ and $\psi_2$ is a locally Lipschitz function.
For a function $u$, we denote its upper and lower semicontinuous envelopes by $u^*$ and $u_*$, respectively.
Let $|\nabla_d u|^*$ be the upper semicontinuous envelope of $|\nabla_d u|$.
% $$
% |\nabla_d u|^*(x) := \limsup_{y \to x} |\nabla_d u|(y).
% $$

\begin{definition}[Metric viscosity solutions]
Let $u: X \to \mathbb{R}$.
\begin{itemize}
\item
$u$ is a \emph{metric viscosity subsolution} of \eqref{e:fhj}
if for every test function $\psi = \psi_1+\psi_2$ for subsolutions such that $u^*-\psi$ attains a local maximum at $\hat{x} \in X$
we have
$$
H_{|\nabla_d\psi_2|^*(\hat{x})}(\hat{x}, u^*(\hat{x}), |\nabla_d\psi_1|(\hat{x})) \le 0,
$$
where $H_a(x, u, p) = \inf_{|p'-p| \le a}H(x, u, p')$.
\item
$u$ is a \emph{metric viscosity subsolution} of \eqref{e:fhj}
if for every test function $\psi = \psi_1+\psi_2$ for supersolutions such that $u_*-\psi$ attains a local minimum at $\hat{x} \in X$
we have
$$
H^{|\nabla_d\psi_2|^*(\hat{x})}(\hat{x}, u_*(\hat{x}), |\nabla_d\psi_1|(\hat{x})) \ge 0,
$$
where $H^a(x, u, p) = \sup_{|p'-p| \le a}H(x, u, p')$.
\item
$u$ is a \emph{metric viscosity solution} of \eqref{e:fhj}
if $u$ is both a metric viscosity subsolution and supersolution.
\end{itemize}
\end{definition}

Under these settings the following theorems are established in \cite{GS15}.

\begin{theorem}[Comparison principle]
Assume (A1), (A2), (A3), and (A4) with $\nu >0$ in \cite{GS15}.
Let $u$ be a metric viscosity subsolution of \eqref{e:fhj} and $v$ be a metric viscosity supersolution of \eqref{e:fhj}.
Let $u, -v$ be bounded from above.
Then, $u \le v$ in $X$ holds. 
\end{theorem}

\begin{theorem}[Unique existence]
Assume $H$ is continuous and there exists $\lambda > 0$ such that $H(x, u, p)-\lambda u$ is non-decreasing.
We also assure
$$
\sup_{x \in X} H(x,0,0) < \infty.
$$
Then, a metric viscosity solution of \eqref{e:fhj} exists.
Furthermore, if we additionally assume (A1), (A3), and (A4) of \cite{GS15}, the metric viscosity solution is unique.
\end{theorem}

Here is a key proposition to show the stability, suggesting that $\psi_2$ can be dropped and $\psi_1$ reduced to the squared distance function in the test functions.

\begin{proposition}
\label{t:sdtest}
Let $(X, d)$ be a locally compact geodesic space.
Then, the following conditions are equivalent.
\begin{itemize}
\item[(i)]
A function $u: X \to \mathbb{R}$ is a metric viscosity subsolution of \eqref{e:fhj}.
\item[(ii)]
For all $\hat{a}, \hat{x} \in X$, $k \ge 0$,
if $u^*(x)-\frac{k}{2}d(\hat{a}, x)^2$ attains a local maximum at $x = \hat{x}$,
then the inequality
$$
H(\hat{x}, u^*(\hat{x}), k d(\hat{a}, \hat{x})) \le 0
$$
holds.
\item[(iii)]
For all $\hat{a}, \hat{x} \in X$, $k \ge 0$,
if $u^*(x)-\frac{k}{2}d(\hat{a}, x)^2$ attains a strict local maximum at $x = \hat{x}$,
then the inequality
$$
H(\hat{x}, u^*(\hat{x}), k d(\hat{a}, \hat{x})) \le 0
$$
holds.
\end{itemize}
Similarly, the following conditions are equivalent.
\begin{itemize}
\item[(i)]
A function $u: X \to \mathbb{R}$ is a metric viscosity supersolution of \eqref{e:fhj}.
\item[(ii)]
For all $\hat{a}, \hat{x} \in X$, $k \ge 0$,
if $u_*(x)+\frac{k}{2}d(\hat{a}, x)^2$ attains a local minimum at $x = \hat{x}$,
then the inequality
$$
H(\hat{x}, u_*(\hat{x}), k d(\hat{a}, \hat{x})) \ge 0
$$
holds.
\item[(iii)]
For all $\hat{a}, \hat{x} \in X$, $k \ge 0$,
if $u_*(x)+\frac{k}{2}d(\hat{a}, x)^2$ attains a strict local minimum at $x = \hat{x}$,
then the inequality
$$
H(\hat{x}, u_*(\hat{x}), k d(\hat{a}, \hat{x})) \ge 0
$$
holds.
\end{itemize}
\end{proposition}

\begin{proof}
We only show the assertion for subsolutions.

To show that condition (ii) implies condition (i), test function $\psi = \psi_1+\psi_2$ for subsolutions should be fixed such that $u^*-\psi$ attains a local maximum at $\hat{x} \in X$.
Consider
$$
(x, y) \mapsto u^*(x)-\psi(y)-\frac{1}{2\e}d(x, y)^2-\frac{\alpha}{2}d(\hat{x}, y)^2
$$
with positive $\e$, $\alpha$.
Since $X$ is locally compact we can take a local maximum point $(x_\e, y_\e)$.
Then,
$$
u^*(x_\e)-\psi(y_\e)-\frac{1}{2\e}d(x_\e, y_\e)^2-\frac{\alpha}{2}d(\hat{x}, y_\e)^2
\ge u^*(\hat{x})-\psi(\hat{x})
\ge u^*(x_\e)-\psi(x_\e)
$$
and hence
$$
\frac{1}{2\e}d(x_\e, y_\e)^2+\frac{\alpha}{2}d(\hat{x}, y_\e)^2 \le L d(x_\e, y_\e)
$$
for the (local) Lipschitz constant $L$ of $\psi$.
Therefore, $(x_\e, y_\e) \to (\hat{x}, \hat{x})$ as $\e \to 0$.
Now, by the assumption of $u$ we have
$$
H(x_\e, u^*(x_\e), \frac{1}{\e}d(x_\e, y_\e)) \le 0.
$$
By Lemma \ref{t:mp} we obtain
$$
||\nabla_d\psi_1|(y_\e)-\frac{1}{\e}d(x_\e, y_\e)| \le |\nabla_d\psi_2|(t_\e, y_\e)+\alpha d(\hat{x}, y_\e).
$$
Combining them,
$$
H_{|\nabla_d\psi_2|(y_\e)+\alpha d(\hat{x}, y_\e)}(x_\e, u^*(x_\e), |\nabla_d\psi_1|(y_\e)) \le 0
$$
and therefore
$$
H_{|\nabla_d\psi_2|^*(\hat{x})}(\hat{x}, u^*(\hat{t}), |\nabla_d\psi_1|(\hat{x})) \le 0.
$$
Thus, $u$ is a metric viscosity subsolution.

Condition (iii) implies condition (ii).
$\hat{a}, \hat{x} \in X$ and $k \ge 0$ should be fixed such that $u^*(x)-\frac{k}{2}d(\hat{a}, x)^2$ attains a local maximum at $x = \hat{x}$.
Hence, for $\eta > 0$ the function $u^*(x)-\frac{k+\eta}{2}d(\hat{a}, x)^2$ attains a strict local maximum at $x = \hat{x}$ and hence
$$
H(\hat{x}, u^*(\hat{x}), (k+\eta)d(\hat{a}, \hat{x})) \le 0.
$$
Sending $\eta \to 0$, we have the desired inequality.

The proof of the other parts of this proposition are trivial.
\end{proof}

\begin{remark}
Proposition \ref{t:sdtest} reveals that metric viscosity solutions and classical viscosity solutions of the Hamilton-Jacobi equation are equivalent when the space is a Euclidean space $\mathbb{R}^d$ with the Euclidean norm $|\cdot|$ and the Euclidean distance $d^E(a, b) = |b-a|$.
Let us illustrate the equivalence with the example of the stationary Hamilton-Jacobi equation of the form
\begin{align}\label{e:shjinE}
    H(x,u(x),|\nabla u(x)|)=0 \quad \text{for $x \in \mathbb{R}^d $} 
\end{align}
where the Hamiltonian $H=H(x,u,p): \mathbb{R}^{d} \times \mathbb{R} \times \mathbb{R}_{+} \to \mathbb{R}$ is continuous.
Function $u$ is a viscosity subsolution of \eqref{e:shjinE} if $u^*-\phi$ with $\phi \in C^{1}(\mathbb{R}^d)$ have a local maximum at $x \in \mathbb{R}^d$. Then
$$
H(x,u^{*}(x),|\nabla \phi(x)|)\le 0.
$$
Similarly, a function $u$ is a viscosity supersolution of \eqref{e:shjinE} if $u_{*}-\phi$ with $\phi \in C^{1}(\mathbb{R}^d)$ have a local minimum at $x \in \mathbb{R}^d$, then
$$
H(x,u_{*}(x),|\nabla \phi(x)|)\ge 0.
$$
Note that $|\nabla \phi (x)|$ is the norm of the gradient $\nabla\phi$,
which equals to the local slope $|\nabla_{d^E}\phi|(x)$.
To prove the equivalence of the viscosity solution and the metric viscosity solution of \eqref{e:shjinE} in Euclidean space,
it is sufficient to show that if property (ii) of Proposition \ref{t:sdtest} for $u$ holds, then the function $u$ is viscosity solution.
The converse is obvious from the definition of the viscosity solution.
Assume $u$ satisfies the following inequality.
For all $\hat{a}, \hat{x} \in \mathbb{R}^d$, $k \ge 0$,
if $u^{*}(x)-\frac{k}{2}d(\hat{a}, x)^2$ attains a local maximum at $x = \hat{x}$,
then the inequality
$$
H(\hat{x}, u^{*}(\hat{x}), k d^{E}(\hat{a}, \hat{x})) \le 0
$$
holds. In the same way in Proposition \ref{t:sdtest}, we consider
$$
(x, y) \mapsto u^{*}(x)-\phi(y)-\frac{1}{2\e}|x-y|^2-\frac{\alpha}{2}|y-\hat{x}|^2
$$
with positive $\e, \alpha$ and this function have a local maximum at some $(x_\e, y_\e) \in \mathbb{R}^d \times \mathbb{R}^d$.
The same calculation yields
$$
H(x_\e, u^{*}(x_\e), \frac{1}{\e}|x_\e-y_\e|) \le 0
$$ 
and $x_\e, y_\e \to \hat{x}$ as $\e \to 0$.
Since, in Euclidean space, the classical maximum principle holds, we get
$$
\nabla \phi(y_\e)-\frac{1}{\e}(x_\e-y_\e)+\alpha(y_\e-\hat{x}) = 0.
$$
Thus we can prove $u$ is a viscosity subsolution of \eqref{e:shjinE}.
Similar conclusions can be obtained for viscosity supersolutions.
In conclusion, we see that the viscosity solution can be characterized using the distance $d^{E}$, as in Proposition \ref{t:sdtest}.
Since, for Euclidean space $\mathbb{R}^d$, the Euclidean distance and the distance induced by the length of the curve are equal, so the viscosity solution and the metric viscosity solution of \eqref{e:shjinE} are equivalent.
\end{remark}

\section{Main results on stability}
\label{s:stability}

Let $(\mathsf{X}, \mathsf{d})$ be a proper metric space and for $n \in \mathbb{N}\cup\{\infty\}$ let $X_n$ be a compact subset of $\mathsf{X}$ equipped with the intrinsic metric $d_n$ from $\mathsf{d}$.
Here we say that $(\mathsf{X},\mathsf{d})$ be a proper if any closed ball in $\mathsf{X}$ is compact.
We assume that $(X_n, d_n)$ forms a geodesic space.
For simplicity, we denote the sequence $x_n \in X_n \to x \in X_\infty$ under the metric $\mathsf{d}$ by $(x_n) \in S(x)$.

\begin{theorem}[Stability]\label{Stability}
Assume \ref{i:h1}, \ref{i:h2}, and \ref{i:h3}.
\begin{itemize}
\item
Let $u^n$ be a sequence of a metric viscosity subsolution of \eqref{e:hj}.
If its upper semilimit 
$$
\overline{u}^\infty(x)
:= \limsup_{n \to \infty, x_n \in X_n \to x \in X_\infty}u^n(x_n)
= \sup_{(x_n) \in S(x)}\limsup_{n \to \infty}u^n(x_n)
$$
is upper semicontinuous in $d_\infty$, 
then $\overline{u}^\infty$ is a metric viscosity subsolution of \eqref{e:lhj}.
\item
Similarly, let $u^n$ is a sequence of a metric viscosity supersolution of \eqref{e:hj}.
If its lower semilimit
$$
\underline{u}^\infty(x)
:= \liminf_{n \to \infty, x_n \in X_n \to x \in X_\infty}u^n(x_n)
= \inf_{(x_n) \in S(x)}\liminf_{n \to \infty}u^n(x_n)
$$
is lower semicontinuous in $d_\infty$,
then $\underline{u}^\infty$ is a metric viscosity supersolution of \eqref{e:lhj}.
\end{itemize}
\end{theorem}
\begin{remark}
The assumption \ref{i:h2} must be essential, because the distance function $u_{n}(x)=d_{n}(a_{n},x)$ is characterized by the solution of the boundary value problem of the Eikonal equation of the form
$$
|\nabla_{d_{n}}u_{n}|(x) = 1 \quad \text{for $x \in X_{n}, x \ne a_n$,} \quad u_{n}(a_{n})=0.
$$
\end{remark}

\begin{proof}
We only show the theorem for the subsolution since the proof for the supersolution is similar.
The upper semilimit $\overline{u}^\infty$ is upper semicontinuous.
Fix $\hat{a}, \hat{x} \in X_\infty$, $k \ge 0$
and assume that $\overline{u}^\infty(x)-\frac{k}{2}d_\infty(\hat{a}, x)^2$ attains a strict local maximum at $x = \hat{x}$ over $X_\infty$.
By the assumption \ref{i:h1}, there exists sequences $a_n, \tilde{x}_n \in X_n$ satisfying $a_n \to \hat{a}$, $\tilde{x}_n \to \hat{x}$ in $\mathsf{d}$ and $u^n(\tilde{x}_n) \to \overline{u}^\infty(\hat{x})$.
For $n \in \mathbb{N}$ consider the function
$$
x \in X_n \mapsto u_n(x)-\frac{k}{2}d_n(a_n, x)^2.
$$
Since this is an upper semicontinuous function on $(X_n, d_n)$,
there is a local maximum point $x = x_n$.
Note that
$$
u^n(\tilde{x}_n)-\frac{k}{2}d_n(a_n, \tilde{x}_n)^2 \le u^n(x_n)-\frac{k}{2}d_n(a_n, x_n)^2
$$
and the left-hand side converges to $\overline{u}^\infty(\hat{x})-\frac{k}{2}d_\infty(\hat{a}, \hat{x})^2$.
Therefore, since $\hat{x}$ is a strict maximum point, $x_n \to \hat{x}$ in $\mathsf{d}$ and $u^n(x_n) \to \overline{u}^\infty(\hat{x})$.
Now, take the limit of
$$
H_n(x_n, u^n(x_n), k d_n(a_n, x_n)) \le 0.
$$
In view of \ref{i:h2} and \ref{i:h3} we have
$$
H_\infty(\hat{x}, \overline{u}^\infty(\hat{x}), k d_\infty(\hat{a}, \hat{x})) \le 0,
$$
which completes the proof.
\end{proof}

\begin{corollary}
Assume \ref{i:h1}, \ref{i:h2}, \ref{i:h3} and that there exists $\lambda > 0$ such that $H_n(x, u, p)-\lambda u$ is non-decreasing for every $n \in \mathbb{N}\cup\{\infty\}$.
Let $u^n$ be a unique metric viscosity solutions of \eqref{e:hj}.
Then, $u^n$ converges to a unique metric viscosity solution $u^\infty$ of \eqref{e:lhj} uniformly.
\end{corollary}

\begin{proof}
Using stability, the upper semilimit $\overline{u}^\infty$ is a metric viscosity subsolution and the lower semilimit $\underline{u}^\infty$ is a metric viscosity supersolution of the limit equation \eqref{e:lhj}.
Now, by the comparison principle, we see that $\overline{u}^\infty \le \underline{u}^\infty$ and therefore $\overline{u}^\infty = \underline{u}^\infty$ is a metric viscosity solution.
\end{proof}

\begin{remark}
    Basically, we mention that the unique existence of viscosity solutions is guaranteed when the comparison theorem and Perron$'$s method can be used. For the detailed conditions required for the Hamiltonian, see \cite{GS15, NN18}. For example, the unique existence of the viscosity solution of (\ref{e:fhj}) is guaranteed for the quadratic Hamiltonian.
\end{remark}

\section{Example: Self-similar sets}
\label{s:ifs}

\subsection{Iterated function systems}

Iterated function systems are self-similar sets in which for a given family of contraction mappings $F_1, \cdots F_N$ with $N = 1, 2, 3, \cdots$,
the Hutchinson operator
$$
\mathcal{F}(K) = F_1(K)\cup\cdots\cup F_N(K)
$$
is a contraction mapping on the space of non-empty compact subsets of $\mathsf{X}$ under the Hausdorff metric $\mathsf{d}_H$.
Hence, a unique non-empty compact set $K$ is a fixed point of the Hutchinson operator $\mathcal{F}$
and that for any non-empty compact set $K_0$ the iteration
$$
K_{n+1} = \mathcal{F}(K_n) = F_1(K_n)\cup\cdots\cup F_N(K_n)
$$
converges to $K$ in the sense of
$$
K = \bigcap_{n=0}^\infty \mathrm{cl}\left(\bigcup_{l=n}^\infty K_l\right),
$$
where $\mathrm{cl}$ means the closure with respect to the topology $\mathsf{d}$.
We call $K_n$ \emph{pre-fractals} of the fractal $K$.
For details on the iterated function systems \cite{F14}.

\subsection{The Vicsek fractal}

Consider the iterated function system on two dimensional Euclidean space $(\mathbb{R}^2, d^E)$ defined by
$$
F_1(x, y) = \left(\frac{1}{3}x, \frac{1}{3}y\right),
\quad F_2(x, y) = \left(\frac{1}{3}x+2, \frac{1}{3}y\right),
\quad F_3(x, y) = \left(\frac{1}{3}x, \frac{1}{3}y+2\right),
$$
$$
F_4(x, y) = \left(\frac{1}{3}x-2, \frac{1}{3}y\right),
\quad F_5(x, y) = \left(\frac{1}{3}x, \frac{1}{3}y-2\right),
$$
and pre-fractals $K_n$ staring from
$$
K_0 = \{ (x, 0) \mid -3 \le x \le 3 \}\cup\{ (0, y) \mid -3 \le y \le 3 \}.
$$
The limit fractal $K$ is called the \emph{Vicsek fractal}.
$K_n$ is embedded in $K$ preserving distance.

\begin{theorem}[Approximation of the Vicsek fractal]
The prefractal approximation $(K_n, d^K_n)$ satisfies condition \ref{i:h2},
i.e.,\ for arbitrary sequences $a_n \in K_n \to a \in K$ and $b_n \in K_n \to b \in K$ under the metric $d^K$,
the convergence
$$
d^K_n(a_n, b_n) \to d^K(a, b)
$$
holds.
\end{theorem}

\begin{proof}
Since $d^K_n = d^K$,
we see that
$$
|d^K_n(a_n, b_n)-d^K(a, b)|
= |d^K(a_n, b_n)-d^K(a, b)|
\le d^K(a, a_n)+d^K(b, b_n)
\to 0.
$$
This completes the proof.
\end{proof}

\subsection{The Sierpi\'{n}ski gasket}

Consider the vertices of the unit triangle on two dimensional Euclidean space $(\mathbb{R}^2, d^E)$ given by
$$
O_1 = \left(-\frac{1}{2}, 0\right),
\quad O_2 = \left(\frac{1}{2}, 0\right),
\quad O_3 = \left(0, \frac{\sqrt{3}}{2}\right).
$$
The \emph{Sierpi\'{n}ski gasket} is generated by the family of contraction mappings
$$
F_i(y) = \frac{1}{2}(y-O_i)+O_i \quad \text{for $i = 1, 2, 3$.}
$$
Consider the intrinsic metric $d^K$ from the Euclidean metric $d^E$.
\begin{lemma}[{\cite[Lemma 2.12]{B98}}]
The metric $d^K$ is equivalent to $d^E$.
Moreover, the inequality
$$
\text{$d^E(a, b) \le d^K(a, b) \le 8 d^E(a, b)$ for all $a, b \in K$}
$$
holds.
\end{lemma}
For details on the metric structure on the Sierpi\'{n}ski gasket,
we refer the readers to \cite{B98, GT98, SOD18, SOD18b}.

There are three prefractal approximations of the Sierpi\'{n}ski gasket $K$.
The first one is graph approximation $V_n$ starting from the vertices $V_0 = \{ O_1, O_2, O_3 \}$.
The second one is Cantor-type approximation $D_n$ starting from with the convex hull $D_0$ of $V_0$, which is the unit triangle.
The last one is network approximation $G_n$ starting from the boundary $G_0$ of $D_0$, which is the boundary of the unit triangle.
One can introduce natural geodesic metric $d^G_n$ and $d^D_n$ into $G_n$ and $D_n$ respectively,
and natural graph structure into $V_n$ with the metric $d^V_n$.
\begin{lemma}
The restriction of the metrics $d^K$, $d^G_n$ and $d^D_n$ to the vertices $V_n$ is nothing but $d^V_n$\rm{:}
$$
\text{$d^K(a, b) = d^G_n(a, b) = d^D_n(a, b) = d^V_n(a, b)$ for all $a, b \in V_n$}
$$
holds.
\end{lemma}
We omit the proof because this is obvious from the explicit formula of the metric of the Sierpi\'{n}ski gasket.

The network and Cantor-type approximations satisfy condition \ref{i:h2}.

\begin{theorem}[Network approximation of the Sierpi\'{n}ski gasket]
The network approximation $(G_n, d^G_n)$ satisfies condition \ref{i:h2}\rm{:} for arbitrary sequences $a_n \in G_n \to a \in K$ and $b_n \in G_n \to b \in K$ under the metric $d^K$,
the convergence
$$
d^G_n(a_n, b_n) \to d^K(a, b)
$$
holds.
\end{theorem}

\begin{proof}
For each $a_n, b_n$ take the nearest point $\tilde{a}_n, \tilde{b}_n$ of $V_n$ in the metric $d^K$.
Note that $d^K(a_n, \tilde{a}_n), d^K(b_n, \tilde{b}_n) \le 2^{-n}$.
We then see that
$$
\begin{aligned}
|d^G_n(a_n, b_n)-d^K(a, b)|
&\le 2\cdot 2^{-n}+|d^G_n(\tilde{a}_n, \tilde{b}_n)-d^K(a, b)| \\
&= 2\cdot 2^{-n}+|d^K(\tilde{a}_n, \tilde{b}_n)-d^K(a, b)| \\
&\le 4\cdot 2^{-n}+d^K(a, a_n)+d^K(b, b_n)
\to 0.
\end{aligned}
$$
This completes the proof.
\end{proof}

\begin{theorem}[Cantor-type approximation of the Sierpi\'{n}ski gasket]
The Cantor-type approximation $(D_n, d^D_n)$ satisfies condition \ref{i:h2}\rm{:}\ for arbitrary sequences $a_n \in D_n \to a \in K$ and $b_n \in D_n \to b \in K$ under the metric $d^E$,
the convergence
$$
d^D_n(a_n, b_n) \to d^K(a, b)
$$
holds.
\end{theorem}

\begin{proof}
For each $a_n, b_n$ take the nearest point $\tilde{a}_n, \tilde{b}_n$ of $V_n$ in the metric $d^E$.
Note that $d^E(a_n, \tilde{a}_n), d^E(b_n, \tilde{b}_n) \le 2^{-n}$.
We then see that
$$
\begin{aligned}
|d^D_n(a_n, b_n)-d^K(a, b)|
&\le 2\cdot 2^{-n}+|d^G_n(\tilde{a}_n, \tilde{b}_n)-d^K(a, b)| \\
&= 2\cdot 2^{-n}+|d^K(\tilde{a}_n, \tilde{b}_n)-d^K(a, b)| \\
&\le 4\cdot 2^{-n}+8 d^E(a, a_n)+8 d^E(b, b_n)
\to 0.
\end{aligned}
$$
This completes the proof.
\end{proof}

\subsection{On the Koch curve}

We remark that every self-similar fractal and its prefractal do not necessarily satisfy the condition \ref{i:h2}.
For instance, consider the Koch curve, which is generated from the iterated function system on the plane
$$
F_1(x, y) = \left(\frac{1}{3}x, \frac{1}{3}y\right),
\quad F_2(x, y) = \left(\frac{1}{6}x-\frac{\sqrt{3}}{6}y+\frac{1}{3}, \frac{\sqrt{3}}{6}x+\frac{1}{6}y\right),
$$
$$
F_3(x, y) = \left(\frac{1}{6}x+\frac{\sqrt{3}}{6}y+\frac{1}{2}, \frac{\sqrt{3}}{6}x-\frac{1}{6}y+\frac{\sqrt{3}}{6}\right),
\quad F_4(x, y) = \left(\frac{1}{3}x+\frac{2}{3}, \frac{1}{3}y\right),
$$
and its pre-fractal starting from
$$
K_0 = \{ (x, 0) \mid 0 \le x \le 1 \}.
$$
Then, the geodesic distance on the prefractal $K_n$ between $(0, 0)$ and $(1, 0)$ can be calculated as
$$
d_n((0, 0), (1, 0)) = \left(\frac{4}{3}\right)^n,
$$
which diverges as $n \to \infty$.

\section{Example: Junction of shrinking tubes}
\label{s:junc}

In order to define a shrinking tube on $\mathbf{R}^2$,
set
$$
e_1 = (1, 0),
\quad e_2 = (-\frac{1}{2}, \frac{\sqrt{3}}{2}),
\quad e_3 = (-\frac{1}{2}, -\frac{\sqrt{3}}{2}),
$$
$$
y_1 = (0, 1),
\quad y_2 = (-\frac{\sqrt{3}}{2}, -\frac{1}{2}),
\quad y_3 = (\frac{\sqrt{3}}{2}, -\frac{1}{2}).
$$
For $i = 1, 2, 3$ define tube parts $R_i$ by
$$
R_i = \{ x = s e_i+t y_i \in \mathbb{R}^2 \mid s \ge \frac{\sqrt{3}}{3n}, -\frac{1}{n} \le t \le \frac{1}{n} \}.
$$
Set
$$
u_1 = -\frac{2\sqrt{3}}{3n}e_3,
\quad u_2 = -\frac{2\sqrt{3}}{3n}e_1,
\quad u_3 = -\frac{2\sqrt{3}}{3n}e_2,
\quad u_4 = u_1,
$$
and define center parts $S_i$ by
$$
S_i = \{ x = s u_i+t u_{i+1} \mid 0 \le s,t \le 1, s+t = 1 \}.
$$
Using three parts $T_i := R_i\cup S_i$, define a shrinking tube $T_n$ by
$$
T_n = T_1\cup T_2\cup T_3.
$$
The shrinking tube $T_{n}$ converges to the Y-junction
$$
J = [0, \infty)e_1 \cup [0, \infty)e_2 \cup [0, \infty)e_3
$$
in the Hausdorff sense where $[0, \infty)e_i = \{ s e_i \mid s \ge 0 \}$.
We are now able to prove that the shrinking tube satisfies condition \ref{i:h2} via projection-like map.

\begin{theorem}
The shrinking tube $(T_n, d^T_n)$ satisfies the condition \ref{i:h2}\rm{:} for arbitrary sequences $a_n \in T_n \to a \in J$ and $b_n \in T_n \to b \in J$ under the Euclidean metric $d^E$,
the convergence
$$
d^T_n(a_n, b_n) \to d^J(a, b)
$$
holds.
\end{theorem}

\begin{proof}
Take joint parts $P_i$ of $T_i$ which is given by
$$
P_i = \{ s u_i \mid 0 < s \le 1 \}.
$$
We construct a projection-like map $F_n : T_n \to J$ as follows.
$$
F_n(x) = (x, e_i)e_i \quad \text{if $x \in T_i\setminus P_i$}.
$$
By the triangle inequality and the property of $F_n$, we see that
\begin{align*}
&|d^T_n(a_n, b_n)-d^J(a, b)| \\
&\quad \le d^T_n(F_n(a_n), a_n)+d^T_n(F_n(b_n), b_n)+|d^T_n(F_n(a_n), F_n(b_n))-d^J(a, b)| \\
&\quad \le 2 n^{-1}+|d^T_n(F_n(a_n), F_n(b_n))-d^J(a, b)|.
\end{align*}
We can evaluate the last term as
\begin{align*}
&|d^T_n(F_n(a_n), F_n(b_n))-d^J(a, b)| \\
&\quad \le d^T_n(F_n(b_n), b)+d^T_n(F_n(a_n), a)+|d^T_n(a, b)-d^J(a, b)| \\
&\quad = d^J(F_n(b_n), b)+d^J(F_n(a_n), a)+|d^T_n(a, b)-d^J(a, b)|.
\end{align*}
Since one can easily check $d^J(F_n(a_n), a) \to 0$, $d^J(F_n(b_n), b) \to 0$, and $|d^T_n(a, b)-d^J(a, b)| \to 0$ as $n \to \infty$,
the proof is now completed.
\end{proof}

\begin{remark}
The situation we considered is a simple case and it must be generalized.
For example, $n^{-1}$-neighborhood
$$
X_n = \{ x \in \mathsf{X} \mid \mathsf{d}(X_\infty, x) \le n^{-1} \}
$$
of a given closed set $X_\infty$ should satisfy condition \ref{i:h2} as long as $X_\infty$ is compact.
However, this generalization is left as an open problem.
\end{remark}

\section{Example: Lattice lines and the Manhattan distance}
\label{s:lattice}

Consider the lattice lines with shrinking spacing $n^{-1}$ for $n = 1, 2, 3, \cdots$ in the plane:
$$
L_n = \left(\bigcup_{i \in \mathbb{Z}}\{ n^{-1}i \}\times\mathbb{R}\right) \cup \left(\bigcup_{j \in \mathbb{Z}}\mathbb{R}\times\{ n^{-1}j \}\right)
$$
and the intrinsic metric $d^L_n$ from the Euclidean metric $d^E$.
$L_n$ converges to the plane $\mathbb{R}^2$ as $n \to \infty$ in the Hausdorff sense
but its metric structure should not be the usual Euclidean one $d^E$.
We instead introduce the so-called Manhattan distance
$$
d^M((a, b), (x, y)) = |x-a|+|y-b| \quad \text{for $(a, b), (x, y) \in \mathbb{R}^2$.}
$$
$(L_n, d^L_n)$ is embedded in $(\mathbb{R}^2, d^M)$ preserving distance.

\begin{theorem}[Lattice lines]
The lattice lines $(L_n, d^L_n)$ satisfies condition \ref{i:h2}\rm{:} for arbitrary sequences $a_n \in L_n \to a \in \mathbb{R}^2$ and $b_n \in L_n \to b \in \mathbb{R}^2$ under the metric $d^M$,
the convergence
$$
d^L_n(a_n, b_n) \to d^M(a, b)
$$
holds.
\end{theorem}

\begin{proof}
Since $d^L_n = d^M$,
we see that
$$
|d^L_n(a_n, b_n)-d^M(a, b)|
= |d^M(a_n, b_n)-d^M(a, b)|
\le d^M(a, a_n)+d^M(b, b_n)
\to 0.
$$
This completes the proof.
\end{proof}

\section{Remarks on time-dependent case}
\label{s:time}

Our stability theory can be applied to the time-dependent equation immediately.
Consider
\begin{equation}
\label{e:thj}
\partial_t u^n(t, x)+H_n(t, x, u^n(t, x), |\nabla_{d_n} u^n|(t, x)) = 0 \quad \text{for $t > 0$ and $x \in X_n$,}
\end{equation}
and
\begin{equation}
\label{e:tlhj}
\partial_t u^\infty(t, x)+H_\infty(t, x, u^\infty(t, x), |\nabla_{d_\infty} u^\infty|(t, x)) = 0 \quad \text{for $t > 0$ and $x \in X_\infty$,}
\end{equation}
where $\partial_t u$ denotes the derivative of the unknown function $u$ in the time variable $t$.
% For simplicity, we will use the notation of the timespace $z = (t, x)$ and $Q_n = (0, \infty)\times(X_n, d_n)$.
Since the time-dependence is to be added to $H_n$,
condition \ref{i:h3} should be modified as follows:
\begin{enumerate}
\renewcommand{\labelenumi}{(H3')}
\renewcommand{\theenumi}{(H3')}
\item
\label{i:h3t}
For arbitrary sequences $t_n \to t_\infty$ in $(0, \infty)$, $x_n \in X_n \to x_\infty \in X_\infty$ under the metric $\mathsf{d}$, $u_n \to u_\infty$ on $\mathbb{R}$ and $p_n \to p_\infty$ on $\mathbb{R}_+$,
the convergence
$$
H_n(t_n, x_n, u_n, p_n) \to H_\infty(t_\infty, x_\infty, u_\infty, p_\infty)
$$
holds.
\end{enumerate}

Let us review the definition of metric viscosity solutions to the time-dependent equation
\begin{equation}
\label{e:tfhj}
\partial_t u(t, x)+H(t, x, u(t, x), |\nabla_d u|(t, x)) = 0 \quad \text{for $t > 0$ and $x \in X$.}
\end{equation}
We call a locally Lipschitz function $\psi(t, x) = \psi_1(t, x)+\psi_2(t, x)$ is a \emph{test function for subsolutions}
if $\psi_1$ and $\psi_2$ are locally Lipschitz and $C^1$ in the time variable $t$ and $\psi_1$ is a $C^{1, -}$ function in the spatial variable $x$.
A locally Lipschitz function $\psi(t, x) = \psi_1(t, x)+\psi_2(t, x)$ is a \emph{test function for supersolutions}
if $\psi_1$ and $\psi_2$ are locally Lipschitz and $C^1$ in the time variable $t$ and $\psi_1$ is a $C^{1, +}$ function in the spatial variable $x$.

\begin{definition}[Metric viscosity solutions to time-dependent problem]
Let $u: (0, \infty)\times X \to \mathbb{R}$.
\begin{itemize}
\item
$u$ is a \emph{metric viscosity subsolution} of \eqref{e:tfhj}
if for every test function $\psi = \psi_1+\psi_2$ for subsolutions such that $u^*-\psi$ attains a local maximum at $\hat{z} = (\hat{t}, \hat{x}) \in (0, \infty)\times X$
we have
$$
\partial_t \psi(\hat{z})+H_{|\nabla_d\psi_2|^*(\hat{z})}(\hat{z}, u^*(\hat{z}), |\nabla_d\psi_1|(\hat{z})) \le 0.
$$
\item
$u$ is a \emph{metric viscosity subsolution} of \eqref{e:tfhj}
if for every test function $\psi = \psi_1+\psi_2$ for supersolutions such that $u_*-\psi$ attains a local minimum at $\hat{z} = (\hat{t}, \hat{x}) \in (0, \infty)\times X$
we have
$$
\partial_t \psi(\hat{z})+H^{|\nabla_d\psi_2|^*(\hat{z})}(\hat{z}, u_*(\hat{z}), |\nabla_d\psi_1|(\hat{z})) \ge 0.
$$
\item
$u$ is a \emph{metric viscosity solution} of \eqref{e:fhj}
if $u$ is both a metric viscosity subsolution and a metric viscosity supersolution.
\end{itemize}
\end{definition}

Time-dependent version of our key proposition (Proposition \ref{t:sdtest}) can be stated as follows.

\begin{proposition}
\label{t:tsdtest}
Let $(X, d)$ be a locally compact geodesic space.
Then, the following conditions are equivalent.
\begin{itemize}
\item[(i)]
A function $u: (0, \infty)\times X \to \mathbb{R}$ is a metric viscosity subsolution of \eqref{e:tfhj}.
\item[(ii)]
For all $\hat{a}, \hat{x} \in X$, $k \ge 0$ and a $C^1$ function $\phi$,
if $u^*(t, x)-\phi(t)-\frac{k}{2}d(\hat{a}, x)^2$ attains a local maximum at $(t, x) = (\hat{t}, \hat{x})$,
then the inequality
$$
\partial_t \phi(\hat{t})+H(\hat{t}, \hat{x}, u^*(\hat{t}, \hat{x}), k d(\hat{a}, \hat{x})) \le 0
$$
holds.
\item[(iii)]
For all $\hat{a}, \hat{x} \in X$, $k \ge 0$ and a $C^1$ function $\phi$,
if $u^*(t, x)-\phi(t)-\frac{k}{2}d(\hat{a}, x)^2$ attains a strict local maximum at $(t, x) = (\hat{t}, \hat{x})$,
then the inequality
$$
\partial_t \phi(\hat{t})+H(\hat{t}, \hat{x}, u^*(\hat{t}, \hat{x}), k d(\hat{a}, \hat{x})) \le 0
$$
holds.
\end{itemize}
Similarly, the following conditions are equivalent.
\begin{itemize}
\item[(i)]
A function $u: (0, \infty)\times X \to \mathbb{R}$ is a metric viscosity supersolution of \eqref{e:tfhj}.
\item[(ii)]
For all $\hat{a}, \hat{x} \in X$, $k \ge 0$ and a $C^1$ function $\phi$,
if $u_*(t, x)-\phi(t)+\frac{k}{2}d(\hat{a}, x)^2$ attains a local minimum at $(t, x) = (\hat{t}, \hat{x})$,
then the inequality
$$
\partial_t \phi(\hat{t})+H(\hat{t}, \hat{x}, u_*(\hat{t}, \hat{x}), k d(\hat{a}, \hat{x})) \ge 0
$$
holds.
\item[(iii)]
For all $\hat{a}, \hat{x} \in X$, $k \ge 0$ and a $C^1$ function $\phi$,
if $u_*(t, x)-\phi(t)+\frac{k}{2}d(\hat{a}, x)^2$ attains a strict local minimum at $(t, x) = (\hat{t}, \hat{x})$,
then the inequality
$$
\partial_t \phi(\hat{t})+H(\hat{t}, \hat{x}, u_*(\hat{t}, \hat{x}), k d(\hat{a}, \hat{x})) \ge 0
$$
holds.
\end{itemize}
\end{proposition}

The proof is basically the same as in the time-independent case.

\begin{proof}
We only show the assertion for subsolutions and show condition (ii) implies condition (i).
Fix a test function $\psi = \psi_1+\psi_2$ for subsolutions such that $u^*-\psi$ attains a local maximum at $\hat{z} = (\hat{t}, \hat{x}) \in (0, \infty)\times X$.
Consider
$$
(t, x, s, y) \mapsto u^*(t, x)-\psi(s, y)-\frac{1}{2\e}|t-s|^2-\frac{1}{2\e}d(x, y)^2-\frac{\alpha}{2}d(\hat{x}, y)^2
$$
with positive $\e$, $\alpha$.
Since $X$ is locally compact we can take a local maximum point $(t_\e, x_\e, s_\e, y_\e)$.
Then,
$$
\begin{aligned}
&u^*(t_\e, x_\e)-\psi(s_\e, y_\e)-\frac{1}{2\e}|t_\e-s_\e|^2-\frac{1}{2\e}d(x_\e, y_\e)^2-\frac{\alpha}{2}d(\hat{x}, y_\e)^2 \\
&\quad \ge u^*(\hat{x})-\psi(\hat{x})
\ge u^*(t_\e, x_\e)-\psi(t_\e, x_\e)
\end{aligned}
$$
and hence
$$
\frac{1}{2\e}|t_\e-s_\e|^2+\frac{1}{2\e}d(x_\e, y_\e)^2+\frac{\alpha}{2}d(\hat{x}, y_\e)^2 \le L |t_\e-s_\e|+L d(x_\e, y_\e)
$$
for the (local) Lipschitz constant $L$ of $\psi$.
Therefore, we can see that $(t_\e, x_\e, s_\e, y_\e) \to (\hat{t}, \hat{x}, \hat{t}, \hat{x})$ as $\e \to 0$.
Now, by the assumption of $u$ we have
$$
\frac{1}{\e}(t_\e-s_\e)+H(t_\e, x_\e, u^*(t_\e, x_\e), \frac{1}{\e}d(x_\e, y_\e)) \le 0.
$$
By Lemma \ref{t:mp} we obtain
$$
||\nabla_d\psi_1|(s_\e, y_\e)-\frac{1}{\e}d(x_\e, y_\e)| \le |\nabla_d\psi_2|(s_\e, y_\e)+\alpha d(\hat{x}, y_\e).
$$
Also note that
$$
\partial_t\psi(s_\e, y_\e)-\frac{1}{\e}(t_\e-s_\e) = 0.
$$
Combining them,
$$
\partial_t\psi(s_\e, y_\e)+H_{|\nabla_d\psi_2|(s_\e, y_\e)+\alpha d(\hat{x}, y_\e)}(t_\e, x_\e, u^*(t_\e, x_\e), |\nabla_d\psi_1|(s_\e, y_\e)) \le 0
$$
and therefore
$$
\partial_t\psi(\hat{t}, \hat{x})+H_{|\nabla_d\psi_2|^*(\hat{t}, \hat{x})}(\hat{t}, \hat{x}, u^*(\hat{t}, \hat{x}), |\nabla_d\psi_1|(\hat{t}, \hat{x})) \le 0.
$$
Hence, $u$ is a viscosity subsolution.
\end{proof}

This implies the following:

\begin{theorem}[Stability to time-dependent problem] Let $(\mathsf{X},\mathsf{d})$ be proper metric space.
Assume \ref{i:h1}, \ref{i:h2}, and \ref{i:h3t}.
Let $u^n$ be a metric viscosity subsolution of \eqref{e:thj} for $n \in \mathbb{N}$.
If the upper semilimit
$$
\overline{u}^\infty(t, x)
:= \limsup_{n \to \infty, t_n \to t, x_n \in X_n \to x \in X_\infty}u^n(t_n, x_n)
= \sup_{t_n \to t, (x_n) \in S(x)}\limsup_{n \to \infty}u^n(t_n, x_n)
$$
is upper semicontinuous in $d_\infty$,
then $\overline{u}^\infty$ is a metric viscosity subsolution of \eqref{e:tlhj}.
Similarly, let $u^n$ is a metric viscosity supersolution of \eqref{e:thj}.
If the lower semilimit
$$
\underline{u}^\infty(t, x)
:= \liminf_{n \to \infty, t_n \to t, x_n \in X_n \to x \in X_\infty}u^n(t_n, x_n)
= \inf_{t_n \to t, (x_n) \in S(x)}\liminf_{n \to \infty}u^n(t_n, x_n)
$$
is lower semicontinuous in $d_\infty$,
then $\underline{u}^\infty$ is a metric viscosity supersolution of \eqref{e:tlhj}.
\end{theorem}

\section{Remarks on noncompact case}
\label{s:noncpt}

Return the time-independent problem and study the possibility to remove the (locally) compactness assumption of $(\mathsf{X}, \mathsf{d})$ in Theorem \ref{Stability}.
As in \cite{NN18}, a tool to consider the function on noncompact domain is the Ekeland$'$s variational principle.

\begin{proposition}[Ekeland$'$s variational principle]
\label{t:ekeland}
Let $(X, d)$ be a complete metric space and let $F: X \to \mathbb{R}$ be an upper (resp.\ lower) semicontinuous function bounded from above (resp.\ below).
Then, for each $\hat{x} \in X$, there exists $\bar{x} \in X$ such that $d(\hat{x}, \bar{x}) \le 1$, $F(\bar{x}) \ge F(\hat{x})$ (resp.\ $F(\bar{x}) \le F(\hat{x})$) and $x \to F(x)-m d(\bar{x}, x)$ attains a strict maximum (resp.\ minimum) at $\bar{x}$ with $m := \sup F-F(\hat{x})$ (resp.\ $m := \inf F-F(\hat{x})$).
\end{proposition}

For details of the Ekeland$'$s variational principle, see \cite[Corollary 11]{E79}.

In noncompact cases, we cannot drop $\psi_2$ but we can reduce $\psi_2$ to the distance function in the test functions.

\begin{proposition}
\label{t:gsdtest}
Let $(X, d)$ be a complete geodesic space not necessarily locally compact.
Then, the following conditions are equivalent.
\begin{itemize}
\item[(i)]
A function $u: X \to \mathbb{R}$ is a metric viscosity subsolution of \eqref{e:fhj}.
\item[(ii)]
For all $\hat{a}, \hat{b}, \hat{x} \in X$, $k, c \ge 0$,
if $u^*(x)-\frac{k}{2}d(\hat{a}, x)^2-c d(\hat{b}, x)$ attains a local maximum at $x = \hat{x}$,
then the inequality
$$
H_c(\hat{x}, u^*(\hat{x}), k d(\hat{a}, \hat{x})) \le 0
$$
holds.
\item[(iii)]
For all $\hat{a}, \hat{b}, \hat{x} \in X$, $k, c \ge 0$,
if $u^*(x)-\frac{k}{2}d(\hat{a}, x)^2-c d(\hat{b}, x)$ attains a strict local maximum at $x = \hat{x}$,
then the inequality
$$
H_c(\hat{x}, u^*(\hat{x}), k d(\hat{a}, \hat{x})) \le 0
$$
holds.
\end{itemize}
Similarly, the following conditions are equivalent.
\begin{itemize}
\item[(i)]
A function $u: X \to \mathbb{R}$ is a metric viscosity supersolution of \eqref{e:fhj}.
\item[(ii)]
For all $\hat{a}, \hat{b}, \hat{x} \in X$, $k, c \ge 0$,
if $u_*(x)+\frac{k}{2}d(\hat{a}, x)^2+c d(\hat{b}, x)$ attains a local minimum at $x = \hat{x}$,
then the inequality
$$
H^c(\hat{x}, u_*(\hat{x}), k d(\hat{a}, \hat{x})) \ge 0
$$
holds.
\item[(iii)]
For all $\hat{a}, \hat{b}, \hat{x} \in X$, $k, c \ge 0$,
if $u_*(x)+\frac{k}{2}d(\hat{a}, x)^2+c d(\hat{b}, x)$ attains a strict local minimum at $x = \hat{x}$,
then the inequality
$$
H^c(\hat{x}, u_*(\hat{x}), k d(\hat{a}, \hat{x})) \ge 0
$$
holds.
\end{itemize}
\end{proposition}

\begin{proof}
We only show the assertion for subsolutions and show the condition (ii) implies the condition (i).
Fix a test function $\psi = \psi_1+\psi_2$ for subsolutions such that $u^*-\psi$ attains a local maximum at $\hat{x} \in X$.
Consider
$$
F(x, y) = u^*(x)-\psi(y)-\frac{1}{2\e}d(x, y)^2-\frac{\alpha}{2}d(\hat{x}, y)^2
$$
with positive $\e$, $\alpha$.
We then see that
$$
\begin{aligned}
F(x, y)
&\le (u^*-\psi)(\hat{x})+\psi(x)-\psi(y)-\frac{1}{2\e}d(x, y)^2 \\
&\le (u^*-\psi)(\hat{x})+L d(x, y)-\frac{1}{2\e}d(x, y)^2 \\
&\le (u^*-\psi)(\hat{x})+\frac{1}{2}\e L^2
< +\infty
\end{aligned}
$$
for the (local) Lipschitz constant $L$ of $\psi$.
Hence, in view of the Ekeland$'$s variational principle, there exists points $x_\e, y_\e \in X$ such that $d(\hat{x}, x_\e), d(\hat{y}, y_\e) \le 1$ and
$$
u^*(x)-\psi(y)-\frac{1}{2\e}d(x, y)^2-\frac{\alpha}{2}d(\hat{x}, y)^2-m_\e d(x_\e, x)-m_\e d(y_\e, y)
$$
attains a strict local maximum at $(x, y) = (x_\e, y_\e)$.
Here,
$$
m_\e = \sup F-F(\hat{x}, \hat{x}) \le \frac{1}{2}\e L^2
$$
and so $m_\e \to 0$ as $\e \to 0$.
Note that
$$
\begin{aligned}
&u^*(x_\e)-\psi(y_\e)-\frac{1}{2\e}d(x_\e, y_\e)^2-\frac{\alpha}{2}d(\hat{x}, y_\e)^2 \\
&\quad \ge u^*(\hat{x})-\psi(\hat{x})-m_\e d(x_\e, x)-m_\e d(y_\e, y) \\
&\quad \ge u^*(x_\e)-\psi(x_\e)-2 m_\e
\end{aligned}
$$
and hence
$$
\frac{1}{2\e}d(x_\e, y_\e)^2+\frac{\alpha}{2}d(\hat{x}, y_\e)^2 \le L d(x_\e, y_\e)+\e L^2
$$
Therefore, we can see that $(x_\e, y_\e) \to (\hat{x}, \hat{x})$ as $\e \to 0$.
Now, by the assumption (ii), we have
$$
H_{m_\e}(x_\e, u^*(x_\e), \frac{1}{\e}d(x_\e, y_\e)) \le 0.
$$
By Lemma \ref{t:mp} we obtain
$$
||\nabla_d\psi_1|(y_\e)-\frac{1}{\e}d(x_\e, y_\e)| \le |\nabla_d\psi_2|(t_\e, y_\e)+\alpha d(\hat{x}, y_\e)+m_\e.
$$
Combining them,
$$
H_{|\nabla_d\psi_2|(y_\e)+\alpha d(\hat{x}, y_\e)+2 m_\e}(x_\e, u^*(x_\e), |\nabla_d\psi_1|(y_\e)) \le 0
$$
and therefore
$$
H_{|\nabla_d\psi_2|^*(\hat{x})}(\hat{x}, u^*(\hat{t}), |\nabla_d\psi_1|(\hat{x})) \le 0.
$$
We have shown that $u$ is a metric viscosity subsolution.
\end{proof}

Let $B(a, r; X)$ denote the closed ball in the metric space $(X, d)$, say,
$$
B(a, r; X) = \{ x \in X \mid d(a, x) \le r \}.
$$
We now can prove the following theorem:

\begin{theorem}[Stability for noncompact case]
Let $(\mathsf{X},\mathsf{d})$ be proper metric space.
Assume \ref{i:h1}, \ref{i:h2}, and \ref{i:h3}.
\begin{itemize}
\item
Assume
\begin{enumerate}
\renewcommand{\labelenumi}{($\text{H4}^+$)}
\renewcommand{\theenumi}{($\text{H4}^+$)}
\item
\label{i:h4p}
For each sequence $(a_n) \in S(a_\infty)$ and upper semicontinuous functions $\phi_n$ on $X_n$ and $\phi_\infty$ on $X_\infty$ such that $\limsup_n \phi_n(x_n) \le \phi_\infty(x_\infty)$ for all $(x_n) \in S(x_\infty)$,
there exists $\delta > 0$ such that the inequality
$$
\limsup_n \sup_{B(a_n, r; X_n)}\phi_n
\le \sup_{B(a_\infty, r; X_\infty)}\phi_\infty
$$
holds for all positive $r < \delta$.
\end{enumerate}
Let $u^n$ be a metric viscosity subsolution of \eqref{e:hj} for $n \in \mathbb{N}$.
If the upper semilimit
$$
\overline{u}^\infty(x)
:= \limsup_{n \to \infty, x_n \in X_n \to x \in X_\infty}u^n(x_n)
= \sup_{(x_n) \in S(x)}\limsup_{n \to \infty}u^n(x_n)
$$
is upper semicontinuous in $d_\infty$,
then $\overline{u}^\infty$ is a metric viscosity subsolution of \eqref{e:lhj}.
\item
Similarly, assume
\begin{enumerate}
\renewcommand{\labelenumi}{($\text{H4}^-$)}
\renewcommand{\theenumi}{($\text{H4}^-$)}
\item
\label{i:h4n}
For each sequence $(a_n) \in S(a_\infty)$ and upper semicontinuous functions $\phi_n$ on $X_n$ and $\phi_\infty$ on $X_\infty$ such that $\liminf_n \phi_n(x_n) \ge \phi_\infty(x_\infty)$ for all $(x_n) \in S(x_\infty)$,
there exists $\delta > 0$ such that the inequality
$$
\liminf_n \inf_{B(a_n, r; X_n)}\phi_n
\ge \inf_{B(a_\infty, r; X_\infty)}\phi_\infty
$$
holds for all positive $r < \delta$.
\end{enumerate}
Let $u^n$ be a metric viscosity supersolution of \eqref{e:hj} for $n \in \mathbb{N}$.
If the lower semilimit
$$
\underline{u}^\infty(x)
:= \liminf_{n \to \infty, x_n \in X_n \to x \in X_\infty}u^n(x_n)
= \inf_{(x_n) \in S(x)}\liminf_{n \to \infty}u^n(x_n)
$$
is lower semicontinuous in $d_\infty$,
then $\underline{u}^\infty$ is a metric viscosity supersolution of \eqref{e:lhj}.
\end{itemize}
\end{theorem}

\begin{proof}
We only show the theorem for the subsolution since the proof for the supersolution is similar.
Fix $\hat{a}, \hat{b}, \hat{x} \in X_\infty$, $k, c \ge 0$
and assume that the function
$$
\phi_\infty(x) = \overline{u}^\infty(x)-\frac{k}{2}d_\infty(\hat{a}, x)^2-c d_\infty(\hat{b}, x) \quad \text{for $x \in X_\infty$}
$$
attains a strict local maximum at $x = \hat{x}$.
Let $R > 0$ be a radius such that
$$
\sup_{x \in B(\hat{x}, R; X_\infty)}\phi_\infty(x)
= \phi_\infty(\hat{x})
$$
By the assumption \ref{i:h1}, there exists sequences $a_n, b_n, \tilde{x}_n \in X_n$ satisfying $a_n \to \hat{a}$, $b_n \to \hat{b}$, $\tilde{x}_n \to \hat{x}$ in $\mathsf{d}$ and $u^n(\tilde{x}_n) \to \overline{u}^\infty(\hat{x})$.
For $n \in \mathbb{N}$ applying the Ekeland$'$s variational principle to the function
$$
\phi_n(x) = u_n(x)-\frac{k}{2}d_n(a_n, x)^2-c d_n(b_n, x) \quad \text{for $x \in X_n$,}
$$
we see that there is a maximum point $x = x_n$ of
$$
x \mapsto \phi_n(x)-m_n d_n(x_n, x) = u_n(x)-\frac{k}{2}d_n(a_n, x)^2-c d_n(b_n, x)-m_n d_n(x_n, x),
$$
where
$$
m_n = \sup_{x \in B(\tilde{x}_n, R; X_n)}\phi_n(x)-\phi_n(\tilde{x}_n).
$$
Assumptions \ref{i:h4p} and \ref{i:h2} implies
$$
\limsup_n m_n
\le \sup_{x \in B(\hat{x}, R; X_\infty)}\phi_\infty(x)-\phi_\infty(\hat{x})
\le 0
$$
and therefore $m_n \to 0$.
Note that
$$
\phi_n(x_n) \ge \phi_n(\tilde{x}_n)-2 R m_n
$$
and the right-hand side converges to $\phi_\infty(\hat{x})$.
Therefore, since $\hat{x}$ is a strict maximum point,
we see that $x_n \to \hat{x}$ in $\mathsf{d}$ and $u^n(x_n) \to \overline{u}^\infty(\hat{x})$.
Now, take the limit of
$$
(H_n)_{c+m_n}(x_n, u^n(x_n), k d_n(a_n, x_n)) \le 0.
$$
In view of \ref{i:h2} and \ref{i:h3}, we have
$$
(H_\infty)_c(\hat{x}, \overline{u}^\infty(\hat{x}), k d_\infty(\hat{a}, \hat{x})) \le 0,
$$
which completes the proof.
\end{proof}

\begin{remark}
\label{t:noth1}
In noncompact case the original definition of Hausdorff convergence \ref{i:h1} may be too strong to handle various problems.
For example, consider the half-lines
$$
X_n = \{ (x, y) \in \mathbb{R}^2 \mid x \ge 0, y = n^{-1}x \},
\quad X_\infty = \{ (x, y) \in \mathbb{R}^2 \mid x \ge 0, y = 0 \}
$$
in the plane $\mathbb{R}^2$.
At the point $P = (x, n^{-1}x) \in X_n$ the distance from $X_\infty$ is given by
$$
d(X_\infty, P) = n^{-1}x
$$
so that the Hausdorff distance $d_H(X_\infty, X_n)$ is always infinity.
In particular, $X_n$ does not converges to $X_\infty$ in the sense of \ref{i:h1}.
Meanwhile, assumption \ref{i:h2} is fulfilled.
There are alternative definitions of Hausdorff convergence to relax the concept for noncompact cases \cite{P05}.
However, we do not touch this subject any more since it is not essential in the present study.
\end{remark}

\section*{Acknowledgments}

The authors would like to thank Qing~Liu for fruitful discussions and insightful suggestions.
The work of the first author was supported by JST SPRING, Grant Number JPMJSP2119.
The work of the second author was supported by JSPS KAKENHI Grant Number 19K14566 and Ginpu Fund Research Grant of Kyoto University.
This work was done while the first author visited Kyoto University
and he is grateful for the kind hospitality.

% \bibliographystyle{amsplain}
% \bibliography{references}

\providecommand{\bysame}{\leavevmode\hbox to3em{\hrulefill}\thinspace}
\providecommand{\MR}{\relax\ifhmode\unskip\space\fi MR }
\providecommand{\MRhref}[2]{%
  \href{http://www.ams.org/mathscinet-getitem?mr=#1}{#2}
}
\providecommand{\href}[2]{#2}

\end{document}